\documentclass[a4paper,12pt]{article}

\title{Ovoids in the cyclic presentation of PG(3,q)}
\author{Kanat Abdukhalikov \\	
Department of Mathematical Sciences, \\
UAE University, PO Box 15551, Al Ain, UAE\\
Email: abdukhalik@uaeu.ac.ae \bigskip  \\  
Simeon Ball \\	
Departament de Matem\`atiques, \\
Universitat Polit\`ecnica de Catalunya, 08034 Barcelona, Spain\\ 
Email: simeon.michael.ball@upc.edu\bigskip  \\ 
Duy Ho \\
Department of Mathematics and Statistics, \\
UiT The Arctic University of Norway, Tromsø
9037, Norway\\
Email: duyho92@gmail.com \bigskip \\
Tabriz Popatia \\	
Departament de Matem\`atiques, \\
Universitat Polit\`ecnica de Catalunya, 08034 Barcelona, Spain\\ 
Email: tabriz.popatia@upc.edu }
\date{ }

\usepackage{amsthm,amsmath,amssymb} 
\usepackage{url} 
\usepackage{cases} 
\begin{document} 

\maketitle

\theoremstyle{plain} 
\newtheorem{lemma}{Lemma}[section] 
\newtheorem{theorem}[lemma]{Theorem}
\newtheorem{corollary}[lemma]{Corollary}
\newtheorem{proposition}[lemma]{Proposition}

\theoremstyle{definition}
\newtheorem{definition}{Definition}[section] 
\newtheorem{remark}{Remark}
\newtheorem{example}{Example}

\newcommand{\eps}{\varepsilon}
\newcommand{\inprod}[1]{\left\langle #1 \right\rangle}
\newcommand{\la}{\lambda} 
\newcommand{\al}{\alpha}
\newcommand{\om}{\omega} 
\newcommand{\gam}{\gamma}
\newcommand{\be}{\beta}
\newcommand{\sig}{\sigma}

\begin{abstract}
We consider the  cyclic presentation of $PG(3,q)$ whose points are in the finite field $\mathbb{F}_{q^4}$ and describe the known ovoids therein.
We revisit the set $\mathcal{O}$, consisting of $(q^2+1)$-th roots of unity in $\mathbb{F}_{q^4}$, and prove that it forms an elliptic quadric within the cyclic presentation of $PG(3,q)$.  Additionally, following the work of Glauberman on Suzuki groups, we offer a new description of Suzuki-Tits ovoids in the cyclic presentation of $PG(3,q)$, characterizing them as the zeroes of a  polynomial over $\mathbb{F}_{q^4}$. 
\end{abstract}


\section{Introduction}

In the projective space $PG(3,q)$ with $q>2$, an \textit{ovoid} is a set of $q^2 + 1$ points meeting every line in at most 2 points. The classical example of an ovoid is an elliptic quadric,  whose points come from a non-degenerate elliptic quadratic form. Ovoids have been an interesting topic in finite geometry with important applications in coding theory as it has been shown (for example in \cite{projectb110,ding2019}) that error-correcting codes from ovoids have optimal parameters.

The only known ovoids in $PG(3,q)$ are the elliptic quadrics, which exist for all $q$, and the Suzuki-Tits ovoids, which exist for $q=2^m$, where $m\ge 3$ is odd. It is well known that every ovoid of $PG(3,q)$, with $q$ odd, is an elliptic quadric  (see for example \cite{hirschfeld1995}). On the other hand, the classification problem for $q$
 even has been resolved only  for $q \le 64$, see \cite{okeefe1994, penttila2022}. 
 For this reason,  our main interest  is the case when $q$ is even.

Usually ovoids are studied in the standard presentation of the projective space $PG(3,q)$. In this paper, we   study ovoids in  another model of  $PG(3,q)$  described   in \cite{ball1999},   following the approach by Jamison \cite{jamison} (see also \cite{bruenfisher}). The point set of this model is identified with a cyclic subgroup of the multiplicative group $\mathbb{F}_{q^4}\backslash\{0\}$. We will refer to this model as \textit{the cyclic presentation} of $PG(3,q)$.

In the first part of the paper, we consider the set $\mathcal{O}$ of $(q^2+1)$-th roots of unity in the finite field $\mathbb{F}_{q^4}$.  In \cite{projectb103}, it was shown that the set $\mathcal{O}$ consists of zeros of a quadratic form and it determines an elliptic quadric in $PG(3,q)$. However, a  presentation of $PG(3,q)$ to which the elliptic quadric $\mathcal{O}$ belongs was not provided. In Theorem \ref{ellipticquadric}, we will show that the set $\mathcal{O}$ is an elliptic quadric in the cyclic presentation of $PG(3,q)$.  
We highlight the recent work on the zeros of the projective polynomial $X^{q+1} + X + a$ by Kim et al.  in \cite{mesnager2021,mesnager2021b}, which provides necessary tools to obtain  our result.

In the remainder of the paper, we  study  Suzuki-Tits ovoids. These ovoids bear this name because they were first described by Tits in \cite{tits1962}  and are stabilized by the Suzuki groups $Sz(q)$. These groups $Sz(q)$, also known as the twisted Chevalley groups of type $^2B_2(q)$, were found by Suzuki \cite{suzuki1960}. A comprehensive  treatment of this construction by Tits can be found in the book by Taylor \cite{taylor1992}. 
We also refer the reader to a recent survey \cite{thas2023} and references therein for more on the historical development and impact of the discovery of Suzuki-Tits ovoids.

We mention briefly the general construction of Suzuki-Tits ovoids in $PG(3,q)$; for the definitions of the terminology, we refer the reader to Subsection 2.3.  Associated with $PG(3,q)$ is a generalized quadrangle $W(q)$. For $q=2^m$, where $m\ge 3$ is odd,
the set of absolute points of a polarity of $W(q)$ forms a Suzuki-Tits ovoid. The standard way to construct a Suzuki-Tits ovoid is to introduce a symplectic basis of the vector space  $\mathbb{F}_{q}^4$ over $\mathbb{F}_q$ and  describe related objects with respect to this basis. 

An alternative approach to the above, which we will use in this paper, is to consider the cyclic presentation of $PG(3,q)$ and  the  associated description of  $W(q)$ described in \cite{ball1999} and \cite{ball2004}.

At the conference Combinatorics 2024, Tao Feng brought our attention to the work of Glauberman \cite{glauberman1996} in 1996 on Suzuki groups. In \cite{glauberman1996}, Glauberman already considered the cyclic presentation of $W(q)$ implicitly and  provided suitable polarities for his investigations on the outer automorphisms of $S_6$. He did not provide a description of the Suzuki-Tits ovoids, however. Perhaps the   reason was that the cyclic presentations of $PG(3,q)$ and $W(q)$ were not described explicitly in the literature until 1999 in \cite{ball1999} for $PG(3,q)$ and 2004 in \cite{ball2004} for $W(q)$. 

Following the work of Glauberman, we determine the absolute points of the polarities given in \cite{glauberman1996} and thereby describe  the Suzuki-Tits ovoids in the cyclic presentation of $PG(3,q)$. As a consequence, we obtain a description of these ovoids as a set of zeroes of a polynomial over $\mathbb{F}_{q^4}$. This is the third description of  Suzuki-Tits ovoids, after the original description by Tits \cite{tits1962} in 1962 and the constructions by Wilson \cite{wilson2009,wilson2013}. 

 The content of the paper is organized as follows. In Section 2, we recall preliminary results from finite geometry and  projective polynomials. Elliptic quadrics in the cyclic presentation of $PG(3,q)$ are considered in Section 3. New results on Suzuki-Tits ovoids are reported in Section 4.

\section{Preliminaries}
Let $q=2^m$. 	Let $E \supset K \supset F$ be a chain of finite fields, $|E|=q^4$, $|K|=q^2$, $|F|=q$.  
	In the sequel, we define the following sets:
	\[
	\mathcal{P}:=\{x \in E \mid x^{q^3+q^2+q+1}=1\}, 
	\]
		\[
	\mathcal{O}:=\{x \in E \mid x^{q^2+1}=1\}, 
	\]
		\[
	S:=\{x \in K \mid x^{q+1}=1\}. 
	\]
Each element $x \in \mathcal{P}$ has a unique decomposition $x = \la u$, where $\la \in S$ and $u \in \mathcal{O}$.

\subsection{The roots of the polynomial $P_a(X)=X^{q+1}+X+a$}

For $a \in E$, let
\[
P_a(X)=X^{q+1}+X+a.
\]

We note that  the more general polynomial forms $X^{q+1} + rX^q + sX + t$ with $s \ne r^q$ and $t \ne rs$ can be transformed into this form by the substitution $X=(s-r^q)^{1/q} X_1-r.$ It is clear that $P_a(X)$ have no multiple roots. 
A comprehensive method for solving the equation $P_a(X)=0$ was recently developed in \cite{mesnager2021,mesnager2021b}.  
Adopting the notation in \cite{mesnager2021}, we define   
$$F(X)=  1+X^q+X^{q^2}.$$  
%
%
%
%
%
%
%
%
From \cite{mesnager2021}, we have the following characterization  of $P_a(X)$ when it has $q+1$ roots.

\begin{lemma}[\cite{mesnager2021}] 
	\label{eqnmesnager1} 
	Let  $a \in E$. Then the polynomial $P_a(X)=X^{q+1}+X+a$ has  $q+1$ zeros in $E$  if and only if $F(a)=0$. In this case, there exists $u \in E \backslash K$ such that $a =\dfrac{(u+u^q)^{q^2+1}}{(u+u^{q^2})^{q+1}}$. Let $c=(u+u^q)^{q-1}$. 
		Then the $q+1$ zeros in $E$ of $P_a(X)$ are $x_0=\dfrac{1}{1+c}$ and 
		$x_\gam=x_0(u+\gam)^{q^2-q}$ for $\gam \in F$.  
\end{lemma}
\subsection{The cyclic presentation of $PG(3,q)$}

In this section we recall the cyclic presentation of $PG(3,q)$ first described explicitly  in \cite[Section 5.1]{ball1999}. 
The point set $\mathcal{P}$ consists of $(q^3 + q^2 + q + 1)$-th roots of unity   in $E$. Planes are given by the zeros of equations of the form
\[
\sig^{q^2+q+1}x^{q^2+q+1} + \sig^{q+1}x^{q+1} + \sig x + 1 = 0,
\]
where $\sig \in \mathcal{P}$.  Lines are given by the zeros of polynomials of the form
\[
L_{\al\be}(x):= x^{q+1}+\al x + \be,
\] 
where $\beta \in \mathcal{P}$ and $\alpha \in E$ such that
\begin{equation} \label{conditionline}
	\al^{q+1} = \be^q+\be^{q^2+q+1}.
\end{equation}
We first note the following.
\begin{lemma} \label{conditionlinecompact} Let $\be \in \mathcal{P}$ and $\al \in E$ satisfying condition \eqref{conditionline}. If $\al \ne 0$, then  $\al^{q^2-1}\be^{q+1}=1$.
\end{lemma}
\begin{proof} We have
\begin{align*}
(\al^{q+1})^{q-1}\be^{q+1} &= (\be^q+\be^{q^2+q+1})^{q-1}\be^{q+1}= \be^{q(q-1)}(1+\be^{q^2+1})^{q-1}\be^{q+1} \\
&= \be^{q^2+1}\dfrac{1+\be^{q^3+q}}{1+\be^{q^2+1}} = \dfrac{\be^{q^2+1}+\be^{q^3+q^2+q+1}}{1+\be^{q^2+1}}=1,
\end{align*}
and the proof follows. 
\end{proof}

 It was shown  in \cite{ball1999} how the lines $L_{\al\be}(x)$ are obtained.  In view of recent results from \cite{mesnager2021}, an alternative way to see that these lines containing points from $\mathcal{P}$ is as follows.
\begin{lemma} \label{pointsofline} Let $\be \in \mathcal{P}$ and $\al \in E$ satisfying condition \eqref{conditionline}. Then the equation 
	\begin{equation} \label{eqna}
		X^{q+1}+\al X + \be =0 
	\end{equation}
has $q+1$ roots in $\mathcal{P}$. 	
\end{lemma}
\begin{proof} 
	
1.	If $\al =0$, then $\beta^{q^2+1}=1$. In particular, there exists $e \in E$ such that $e^{q^2-1}=\beta$.  Then the set of solutions of \eqref{eqna} is  
$
 \{ se^{q-1} \mid s \in S\}, 
$
	and each element $se^{q-1}$ is in $\mathcal{P}$.  
	
2. Consider the case $\al \ne 0$. By substituting $X=\al^{q^3}Y$,  the equation \eqref{eqna}  can be transformed into
\begin{equation} \label{eqnb}
	Y^{q+1}+Y+a=0,
\end{equation}
where $a:= \be \al^{-(q^3+1)}$. From Lemma \ref{conditionlinecompact}, we have $\al^{q^2-1}=\be^{q^3+q^2}$, so that  
\[
\al^{q^3+1}=\al^{q^3-q}\al^{q+1}=\be^{q^3+1}\al^{q+1}= \be^{q^3+1}(\be^q+\be^{q^2+q+1}) = \be^{q^3+q+1}(1+\be^{q^2+1}). 
\]
Then 
\[
a = \dfrac{\be}{\al^{q^3+1}}= \dfrac{\be}{\be^{q^3+q+1}(1+\be^{q^2+1})}= \dfrac{1}{\be^{q^3+q}(1+\be^{q^2+1})}= \dfrac{\be^{q^2+1}}{1+\be^{q^2+1}}.
\]
 And so $a \in K$.

Following Subsection 2.1, we let  $F(a)= a^{q^2}+a^q+1.$  It follows readily that $F(a)=0$. 
By Lemma \ref{eqnmesnager1},  there exists $u \in E \backslash K$ such that $a =\dfrac{(u+u^q)^{q^2+1}}{(u+u^{q^2})^{q+1}}$. 
Let $c = (u+u^q)^{q-1} \in \mathcal{P}$. 
The  roots of the equation  \eqref{eqnb} are  given by
$ 
y_0=\dfrac{1}{1+c},
$ 
and
$ 
y_0(u+\gam)^{q^2-q},
$ 
for each $\gam \in F$.

The  roots of the original equation \eqref{eqna} are then
$ 
x_0=\dfrac{\al^{q^3}}{1+c},
$ 
and
$ 
x_0(u+\gam)^{q^2-q},
$ 
for each $\gam \in F$.  
In \cite[p.516]{projectb105}, it was shown that 	
\[
a= \frac{c^q}{(1+c)^{q+1}},
\]
so that
\[
a(1+c)^2 = \frac{c^q}{(1+c)^{q-1}} \in \mathcal{P}. 
\]
We have
\[
x_0^2= \frac{\al^{2q^3}}{(1+c)^{2}} =\frac{a\al^{2q^3} }{a(1+c)^{2}} = \frac{\al^{q^3-1}\be}{a(1+c)^{2}}.
\]
We note that $\al^{q^3-1}\be$ is a $(q-1)$-th power and so it is in $\mathcal{P}$. 
It follows that $x_0 \in \mathcal{P}$, and since
\[
(u+\gam)^{q^2-q}= (u^q+\gam)^{q-1},
\]
all the $q+1$ roots of \eqref{eqna} are also in $\mathcal{P}$. 
\end{proof}

\begin{lemma}[\cite{ball1999}]  \label{connectingline} The line joining two points $y$ and $z$ is given by the zeroes of the polynomial $X^{q+1}+\alpha X +\beta=0$, where
	\[
	\alpha = \dfrac{y^{q+1}+z^{q+1}}{y+z},
	\]
\[
\beta= \dfrac{y^{q+1}z+z^{q+1}y}{y+z}.
\]
	
\end{lemma}

%

\subsection{Known results on the Suzuki-Tits ovoid}

We  collect some materials from \cite[p.41]{dembowski}, \cite{ball2004} and \cite{thas2015}.

A \textit{correlation} of $PG(3,q)$ is a bijection $\sigma$ of the point set of $PG(3,q)$ onto the plane set of $PG(3,q)$, such that the $q + 1$ points of any line $l$ are mapped  onto the $q+1$ planes containing a line $l'$. The line $l'$ is the image of $l$ under $\sigma$. The point set of any plane $P$ of $PG(3,q)$ is mapped onto the set of all planes containing the point $\sigma^{-1}(P)$. 
A \textit{polarity} of $PG(3,q)$ is a correlation of order $2$.  A polarity $\pi$ is  \textit{symplectic} if and only if every point $p$ of  $PG(3,q)$  is absolute, that is $p \in \pi(P)$.  
A  line $l$ is \textit{totally isotropic} (also known as self-polar  in \cite{thas2015}) with respect to a correlation $\sigma$ if $\sigma(l)=l$.

Let $W (q)$ be the point-line geometry formed by all points of $PG(3,q)$ and all totally isotropic lines with respect to a symplectic polarity of $PG(3,q).$  The geometry $W(q)$ is called a tactical configuration in \cite{dembowski} but  otherwise is more commonly known as a \textit{generalized quadrangle} of order $(q,q)$, see for example \cite{ball2004}.

A \textit{correlation} of $W(q)$ is   a bijection $\al$ of $\mathcal{P}\cup\mathcal{L}$ onto itself such that $\al(\mathcal{P})=\mathcal{L}$, $\al(\mathcal{L})=\mathcal{P}$,
$p \in l$ if and only if $\al(l) \in \al(p)$ for $p \in \mathcal{P}$ and $l \in \mathcal{L}$. Such a correlation $\alpha$ is a \textit{polarity} if $\al^2=1$. A point $p$ of $\mathcal{P}$ is \textit{absolute} with respect to $\al$ if $p \in \al(p)$.

\begin{theorem}[\cite{tits1962}] \label{tits}  Suppose that $q=2^m$. The geometry $W (q)$ admits a polarity $\alpha$ if and only if $m$ is odd. In such  case, the absolute points of $\alpha$   form an ovoid $\mathcal{T}$ of $PG(3,q).$ These ovoids are called Suzuki-Tits ovoids. A Suzuki-Tits
ovoid is an elliptic quadric if and only if $q = 2$. 
\end{theorem}

\section{Elliptic quadrics in the cyclic  presentation of $PG(3,q)$}

We recall from \cite{projectb103} the following description of elliptic quadrics in $PG(3,q)$. 
 \begin{theorem}  
	\label{quadric}
	Let $E \supset K \supset F$ be a chain of finite fields, $|E|=q^4$, $|K|=q^2$, $|F|=q$, $q=2^m$.  Then 
	$$Q(x)= Tr_{K/F}(N_{E/K}(x))$$ 
	is a non-degenerate quadratic form on $4$-dimensional vector space  $E$ over $F$. 
	Moreover,  the set 
	$$\mathcal{O}=\{u\in E \mid N_{E/K}(u)=1\}=\{u\in E \mid u^{q^2+1}=1\}$$ 
	determines an elliptic quadric in $PG(3,q)$. 
\end{theorem} 

It is readily checked that $\mathcal{O} \in \mathcal{P}$. In \cite{projectb103},  a  presentation of $PG(3,q)$ to which the elliptic quadric $\mathcal{O}$ belongs was not provided. Here we show that $\mathcal{O}$ is an ovoid  in the cyclic presentation of $PG(3,q)$.
 \begin{theorem} \label{lineintersectionquadric} \label{ellipticquadric} Let $\be \in \mathcal{P}$ and $\al \in E$ satisfying condition \eqref{conditionline}. Let $\la \in S$. Then the   system $(I)$ of equations
    \begin{numcases}{}
    X^{q+1}+\al X+ \beta=0  \label{eqn1}
   \\
   X^{q^2+1}=\la \label{eqn2}
\end{numcases}
 has at most two solutions in $\mathcal{P}$. 
\end{theorem} 
\begin{proof}
%
%
%
We follow the notation in the proof of Lemma \ref{pointsofline}. 
The $q+1$ roots of \eqref{eqn1} are in $\mathcal{P}$ and are  given by
$ 
x_0=\dfrac{\al^{q^3}}{1+c}
$ 
and
$ 
x_0(u^q+\gam)^{q-1} 
$ 
for each $\gam \in F$.

1. Suppose that there exists $\gamma \in F$  such that $x_0(u^q+\gam)^{q-1}$ is a solution of \eqref{eqn2}, that is, 
\[
x_0^{q^2+1}(u^q+\gam)^{(q-1)(q^2+1)}=\la. 
\]
Let
\[
s:= \dfrac{\la}{x_0^{q^2+1}} = (u^q+\gam)^{(q-1)(q^2+1)}=\dfrac{(u^{q^2}+\gam)^{q^2+1}}{(u^q+\gam)^{q^2+1}}.
\]
We have that $(u^q+\gam)^{q^2+1} \ne 0$. Then 
\begin{align*}
	 &(u^q+\gam)^{q^2+1}s  = (u^{q^2}+\gam)^{q^2+1} \\
	\iff  & (u^{q^3}+\gam) (u^{q }+\gam)s = (u^{q^2}+\gam) (u+\gam)\\
	\iff  & A\gam^2+B\gam+C=0,
\end{align*}
where 
\[A=s+1,
\]
\[
B=s(u^{q^3}+u^q)+u^{q^2}+u,
\]
\[
C=su^{q^3+q}+u^{q^2+1}.
\] 

2. The calculation above shows that $x_0(u^q +\gam)^{q-1}$ is   a solution of $(I)$    if and only if $\gamma$ is a solution of the equation 
\begin{equation} \label{quad}
AX^2+BX+C=0.
\end{equation}

If $s=1$, then $A=0$, and \eqref{quad} has a unique solution $\gam_0$. Then $(I)$ has two solutions, which are $x_0$ and $x_0(u^q+\gamma_0)^{q-1}$.

If $s\ne1$, then \eqref{quad} is an irreducible quadratic and has at most two solutions. Furthermore, since $s=\la/x_0^{q^2+1}$, we see that $x_0$ is not a solution of $(I)$.  It follows that  $(I)$  has at most two solutions.  
\end{proof}

\section{Suzuki-Tits ovoids}
\subsection{Polarities in the cyclic presentation of $PG(3,q)$}

 Let $q=2^m$, where $m\ge 3$ is odd. We still consider the cyclic presentation of  $PG(3,q)$ and follow the terminology introduced in Subsection 2.3. 
For every point $\la$ of $\mathcal{P}$, let
\[
\om_\lambda(x):=(\la^{q^2}x)^{q^2+q+1}+(\la^{q^2}x)^{q+1}+ (\la^{q^2}x)+1
\]
be a polynomial over  $\mathbb{F}_{q^4}$ whose zeros correspond to the points of a plane.  Then  a symplectic polarity   $\om$  of $PG(3,q)$ was given in \cite{ball1999} as $\om: \la \mapsto \om_\la(x)=0$. 

Totally isotropic lines with respect to $\om$ are those of the form
\[
X^{q+1} + (\be^{(q^2+q+2)/2} + \be^{(q+1)/2})X + \be = 0,
\]
for each $\be \in \mathcal{P}$. This was proved in \cite{ball2004}. We denote each such line parameterized by $\be$ as $L(\be)$. Let $\mathcal{L}$ be the set of all such totally isotropic lines.  Let $W (q)= (\mathcal{P},\mathcal{L})$ be the corresponding generalized quadrangle. 

This presentation of the generalized quadrangle $W(q)$ was already considered by Glauberman in \cite{glauberman1996}, where it was refered to as the  ``symplectic geometry $\mathcal{G}$ on $E$''. However, the totally isotropic lines were not described explicitly by Glauberman. 

In \cite[p. 64]{glauberman1996}, a correlation 
$\delta$ of $W(q)$ was defined with respect to   a symplectic basis of $E$ over $F$. Furthermore, it was proved (also in \cite[p. 64]{glauberman1996}) that this map $\delta$ can be defined equivalently by the following.

\begin{definition}[The map $\delta$] \label{themapdelta} Let $\delta: W(q) \rightarrow W(q)$ be defined as follows:
\begin{enumerate}
\item $\delta(x)= L(x^{2q})$ for every point $x \in \mathcal{P}$,
\item $\delta(L(x)) =x$ for every line $L(x) \in \mathcal{L}$.
\end{enumerate}
\end{definition}

\begin{remark} It was pointed out in \cite[p.6]{ball2004} that if the point $x$ lies on $L(e)$ then $e$ lies on $L(x^{2q})$. This provides an alternative way to show that $\delta$ is a correlation of $W(q)$.
\end{remark}

\begin{definition}[The maps $\sigma_i$ and $c(\sigma)$] For each  nonnegative integer $i$, let $\sigma_i$ be the automorphism $x \mapsto x^{2^i}$ of $E$. 
For each element $\sigma$  of $Gal(E)$, let $c(\sigma)$ denote the collineation of $W(q)$ given by 
$$W^{c(\sigma)} = \{\sigma(w) \mid w \in W\},$$ 
for every point or line $W$. 
\end{definition}

In \cite[Proposition 3.6]{glauberman1996}, Glauberman proved the following.
\begin{lemma}
Let $\theta_0 = \sigma_{\sqrt{2q}} $ and $\theta_1=  \sigma_{q^2\sqrt{2q}}$. Then $\pi_i := c(\theta_i)^{-1} \delta$ are polarities of  $W(q)$. 
\end{lemma}  

\begin{remark}Using the relation $x^{1/q^2}=x^{q^2}$ for $x \in E$, we can describe $\pi_0$ and $\pi_1$ explicitly.
The map $\pi_0: W(q) \rightarrow W(q)$ is defined  as follows:
\begin{enumerate}
\item $\pi_0(x)= L(x^{\sqrt{2q}})$ for every point $x \in \mathcal{P}$,
\item $\pi_0(L(x)) =x^{1/\sqrt{2q}}$ for every line $L(x) \in \mathcal{L}$.
\end{enumerate} 
The map $\pi_1: W(q) \rightarrow W(q)$ is defined as follows:
\begin{enumerate}
\item $\pi_1(x)= L(x^{q^2\sqrt{2q}})$ for every point $x \in \mathcal{P}$,
\item $\pi_1(L(x)) =x^{q^2/\sqrt{2q}}$ for every line $L(x) \in \mathcal{L}$.
\end{enumerate} 
\end{remark}

\subsection{Absolute points of polarities}
Recall that $q=2^m$, where $m\ge 3$ is odd. Let $s:=q-\sqrt{2q}+1,$ $t:=q+\sqrt{2q}+1$.  We define
$
\mathcal{O}_s:=\{x \in E \mid x^{s}=1\}, 
$
and
$
\mathcal{O}_t:=\{x \in E \mid x^{t}=1\}. 
$ Then $q^3+q^2+q+1=(q+1)(q^2+1)=(q+1)st$,   $\mathcal{O}=\mathcal{O}_s \mathcal{O}_t$,
and every element $x \in \mathcal{P}$ can be decomposed uniquely as $x=\la u v, \la \in S, u \in \mathcal{O}_s, v \in \mathcal{O}_t$. 


\begin{lemma} \label{lary} Let $x \in \mathcal{O}_s \cup \mathcal{O}_t$ such that $x \ne 1$. Let $\la = \left(x^{q-1}+\dfrac{1}{x^{q-1}} \right)^{q-1}.$ Then
\begin{equation*}  
x^{q+1}+\sqrt{(\la^{\sqrt{2q}}+1)\la} x + \la^{\sqrt{2q}/2}=0. 
\end{equation*}
\end{lemma}

\begin{proof} We first note that $\la \in S$. By Lemma \ref{connectingline}, the equation that  describes the line going through $x$ and $1/x$ is given by $X^{q+1}+\al X +\be =0$, where  $\beta=\left( x+\dfrac{1}{x} \right)^{q-1}$ and  
	\[
\alpha = \dfrac{x^{q+1}+1/x^{q+1}}{x+1/x}.
\]
Since $x \in \mathcal{O}_s \cup \mathcal{O}_t$, either $x^{q-1}= x^{q\sqrt{2q}}$ or $x^{q-1}= x^{-q\sqrt{2q}}$. This implies that $\la=\be^{q\sqrt{2q}}$ and so $\beta = \la^{\sqrt{2q}/2}$.
It remains to prove that $\al^2 = (\la^{\sqrt{2q}}+1)\la$. For convenience, we denote   $x^n+1/x^n$ by $D_n$. 
We have
\[
D_{q+1}^{q+1} = D_{q+1}^q D_{q+1} = D_{q^2+q}D_{q+1}=D_{q-1}D_{q+1} = D_1^{2q}+D_1^2,
\]
\[
\left(\be+\be^q\right)D_1^{q+1} = (D_1^{q-1}+D_1^{q^2-q})D_1^{q+1} = D_1^{2q}+D_1^{2}. 
\]
This implies that $D_{q+1}^{q+1}=\left(\be+\be^q\right)D_1^{q+1}$, so that
 \begin{equation} \label{alpha1}
\alpha^{q+1}=\left(\dfrac{x^{q+1}+1/x^{q+1}}{x+1/x}\right)^{q+1} = \left(\dfrac{D_{q+1}}{D_1}\right)^{q+1}= \be+\be^q.
\end{equation}  
On the other hand, since $x \in \mathcal{O}_s \cup \mathcal{O}_t$, either $x^{q+1}= x^{\sqrt{2q}}$ or $x^{q+1}= x^{-\sqrt{2q}}$, and so
\begin{equation} \label{alpha2}
	\alpha^{q-1}=\left(\dfrac{x^{q+1}+1/x^{q+1}}{x+1/x}\right)^{q-1} = \left(\dfrac{D_{1}^{\sqrt{2q}}}{D_1}\right)^{q-1}=\be^{\sqrt{2q}-1}. 
\end{equation}
 From \eqref{alpha1} and \eqref{alpha2}, it follows that
$
 \al^2= \al^{q+1}/\al^{q-1} =(\be+\be^q) \be^{1-\sqrt{2q}}.
$
From the relations $\la=\be^{q\sqrt{2q}}$ and $\beta = \la^{\sqrt{2q}/2}$, we obtain that 
\[
\al^2 = (\be+\be^q) \be^{1-\sqrt{2q}} = (\la^{\sqrt{2q}/2}+\la^{q\sqrt{2q}/2}) \la^{\sqrt{2q}/2} \la = (\la^{\sqrt{2q}}+1)\la,
\]
and the proof follows.
\end{proof}

We now introduce two sets.  
Let 
\[
\mathcal{T}_0:=\mathcal{O}_s \cup \left\{\left(v^{q-1}+\dfrac{1}{v^{q-1}}\right)^{q-1}uv \mid u \in \mathcal{O}_s, v \in \mathcal{O}_t \backslash \{1\}\right\},
\]
and 
\[
\mathcal{T}_1:=\mathcal{O}_t\cup \left\{\left(u^{q-1}+\dfrac{1}{u^{q-1}}\right)^{q-1}uv \mid u \in \mathcal{O}_s \backslash \{1\}, v \in \mathcal{O}_t \right\}. 
\]

\begin{theorem} \label{abspts} For $i\in\{0,1\}$, the set $\mathcal{T}_i$ is the set of absolute points of the polarity $\pi_i$. In particular, $\mathcal{T}_i$ is a Suzuki-Tits ovoid. 
\end{theorem}
\begin{proof}  We will prove the theorem for the case $i=0$, as the case   $i=1$ is similar. Since $|\mathcal{T}_0|=q^2+1$, it is sufficient to show that every point of $\mathcal{T}_0$ is an absolute point of  $\pi_0$, that is $x \in L(x^{\sqrt{2q}})$ for every $x \in \mathcal{T}_0$. 
Let such $x$ be of the form $x=\la u v$, where $\la \in S, u \in \mathcal{O}_s, v \in \mathcal{O}_t$. We have
\[
 A: = (x^{\sqrt{2q}})^{(q^2+q+2)/2} + (x^{\sqrt{2q}})^{(q+1)/2}  = \left( x^{q+1}(x^{q^2+1}+1) \right)^{\sqrt{2q}/2} 
 = \left( \dfrac{u}{v} \right)^q (\la+1)^{\sqrt{2q}}. 
\]
Then
\begin{align*}
x \in L(x^{\sqrt{2q}}) &\iff x^{q+1} + Ax + x^{\sqrt{2q}}= 0\\
& \iff (uv)^{q+1}+\left( \dfrac{u}{v} \right)^q (\la+1)^{\sqrt{2q}} (\la u v) + (\la u v)^{\sqrt{2q}}=0 \\
& \iff v^{2(q+1)}+(\la^{\sqrt{2q}}+1)\la v^2 + \la^{\sqrt{2q}}=0 \\
& \iff v^{q+1}+\sqrt{(\la^{\sqrt{2q}}+1)\la} v + \la^{\sqrt{2q}/2}=0. 
\end{align*}
The last line is clearly true for $x \in \mathcal{O}_s$ and is also true for $x \in \mathcal{T}_0 \backslash \mathcal{O}_s$, by Lemma \ref{lary}. 
\end{proof}

\subsection{The sets $\mathcal{T}_i$ as sets of zeroes of polynomials}
We now show that $\mathcal{T}_0$ is the set of zeroes of a short polynomial over $\mathbb{F}_{q^4}$.
\begin{theorem} \label{polynomialT0}  The set   $\mathcal{T}_0$ is the set of solutions of the equation $Q_0(x)=0$, where
\[
Q_0(x)=x^{q^2+1}+x^{s(\sqrt{2q}+1)}+x^s+1.
\] 
\end{theorem}
\begin{proof}  Let $x= \la uv\in \mathcal{P}$, where $\la \in S, u \in \mathcal{O}_s, v \in \mathcal{O}_t$. Then $x^{q^2+1}=\la^2$, $x^{s(\sqrt{2q}+1)}=\la^{s+2} (v^s)^{\sqrt{2q}+1}$ and $x^s=\la^sv^{s}$. This means that 
 \begin{align*}
Q_0(x)=0 &\iff  \la^{s+2} (v^s)^{\sqrt{2q}+1}+\la^s v^s +\la^2+1=0\\
&\iff v^{s\sqrt{2q}} + (\la^2+1)\la^{(s+2)q}v^{-s}+\la^{2q}=0 \\
&\iff v^{2\sqrt{2q}(q+1)} + (\la^2+1)\la^{(2-\sqrt{2q})q}v^{2\sqrt{2q}}+\la^{2q}=0 \\ 
&\iff v^{2\sqrt{2q}(q+1)} + (1+\la^{2q})\la^{\sqrt{2q}}v^{2\sqrt{2q}}+\la^{2q}=0.
\end{align*}
Here we used the relations $v^{q+1}=1/v^{\sqrt{2q}},$ $v^s=v^{2(q+1)}=1/v^{2\sqrt{2q}}, \la^{\sqrt{2q}}=1/\la^s$.
Taking power $1/\sqrt{2q}$ on both sides of the last line, one obtains
\[
 v^{2(q+1)}+(\la^{\sqrt{2q}}+1)\la v^2 + \la^{\sqrt{2q}}=0.
\]
From the proof of  Theorem \ref{abspts}, we see that each element of $\mathcal{T}_0$ is a solution of $Q_0(x)=0$.   We finally note that $\deg Q_0(x)=|\mathcal{T}_0| =q^2+1$ and this proves the theorem.
\end{proof}



We now describe the polynomial of degree $q^2+1$ whose roots are the elements of $\mathcal{T}_1$. Let us define the line $\ell_{e}$ by the points $x \in \cal P$ such that
\[
x^{q+1} + (e^{q^2+q+2} + e^{q+1})x + e^2 = 0.
\]
By raising this equation to the $q$ and rearranging we deduce that
$$
x \in \ell_{e} \ \Leftrightarrow \ e \in \ell_{x^{q/2}}.
$$
Let $O$ be an ovoid and define a bijective map $\tau:  O \rightarrow  O$ by
$$
\tau(x) \in \ell_x
$$
for all $x \in  O$.

In the following lemma we add the hypothesis that $x\in  O$ implies $x^{q/2} \in O$ for all $x \in   O$. This is equivalent to insisting that $ O$ is Frobenius fixed ($x\in  O$ implies $x^{2} \in O$), which is equivalent to the polynomial of degree $q^2+1$ whose zeros are the points of $ O$ being an element of ${\mathbb F}_2[X]$.

\begin{lemma} \label{fflemma}
If the points of $ O$ are Frobenius fixed, then
$$
\tau(\tau(x)^{q/2})=x
$$
for all $x \in   O$.
\end{lemma}

\begin{proof}
By definition, $x \in \ell_{\tau^{-1}(x)}$. Since $\tau(x) \in \ell_x$, we have that $x \in \ell_{\tau(x)^{q/2}}$. If $  O$ is Frobenius fixed then $\tau(x) \in  O$ implies $\tau(x)^{q/2} \in  O$. Since, there is a unique line of the spread dual to $  O$ incident with $x$, it follows that $\tau(x)^{q/2}=\tau^{-1}(x)$. 
\end{proof}

\begin{example}
The map $\tau(x)=x^{\sqrt{2/q}}$ corresponds to the Tits ovoid ${\cal T}_0$, the map $\tau(x)=x^{q^2\sqrt{2/q}}$ corresponds to the Tits ovoid ${\cal T}_1$ and the map $\tau(x)=x^{1-q}$ gives the canonical elliptic quadric which are zeros of $x^{q^2+1}+1$. It is clear that the first two examples satisfy $\tau(\tau(x)^{q/2})=x$. To verify the claim for $\tau(x)=x^{1-q}$ observe that $\tau(x) \in \ell_x$ implies that
$$
x^{1-q^2}+x^{2}(x^{q^2+1}+1)+x^2=x^{1-q^2}(1+x^{q^2+1})^2=0
$$ 
and $\tau(\tau(x)^{q/2})=x$ implies $x^{(1-q)^2q/2}=x$ implies $x^{3(q^2+1)}=1$.
\end{example}


In the following theorem $\log$ refers to $\log_2$.

\begin{theorem} \label{polynomialT1}  The set   $\mathcal{T}_1$ is the set of solutions of the equation $Q_1(x)=0$, where
\[
Q_1(x)=x^{q^2+1}+1+x^t \left(\frac{1+x^{\sqrt{2q} t(\sqrt{q/2}-1)}}{1+x^{\sqrt{2q} t}}\right)+x^t\sum_{j=0}^{\log \sqrt{q/2} -1} x^{2^j(\sqrt{2q}-2)t}
(1+x^{\sqrt{2q} t})^{2^j-1}.
\] 
\end{theorem}

\begin{proof}
Mimicking the proof of Theorem \ref{polynomialT0}, one obtains
$$
x^{q^2+tq+t+1}+x^{q^2+t+1}+x^t+1=0
$$
which implies
\begin{equation} \label{t1:0}
x^{2qt}+x^{qt}+x^{(\sqrt{2q}-1)t}+x^{(\sqrt{2q}-2)t}=0.
\end{equation}
Consecutively squaring this equation and summing, we get
\begin{equation} \label{t1:1}
x^{\sqrt{2q}qt}+x^{qt}+\sum_{j=0}^{\log \sqrt{q/2}} (x^{2^j(\sqrt{2q}-1)t}+x^{2^j(\sqrt{2q}-2)t})=0.
\end{equation}

This is the case that $\tau(x)=x^{q^2\sqrt{2/q}}$. Thus, $\tau(x) \in \ell_x$ implies $\tau(x)^{\sqrt{2q}} \in \ell_{x^{\sqrt{2q}}}$, which, multiplying by $x^{2q+2}$, implies
\begin{equation} \label{t1:2}
x^{\sqrt{2q}qt+2t}+x^{(2q-\sqrt{2q}+2)t}+x^{2t}+1=0.
\end{equation}
Summing $x^{(-\sqrt{2q}t+2t)}$ times (\ref{t1:0}) gives
\begin{equation} \label{t1:3}
x^{\sqrt{2q}qt+2t}+x^{(q-\sqrt{2q}+2)t}+x^{2t}+x^{t}=0.
\end{equation}
Now, $x^t(\ref{t1:1})+x^{-t}(\ref{t1:3})$ gives
$$
x^{(q-\sqrt{2q}+1)t}+x^{t}+1+x^{(q+1)t}+x^t\sum_{j=0}^{\log \sqrt{q/2}} (x^{2^j(\sqrt{2q}-1)t}+x^{2^j(\sqrt{2q}-2)t})=0.
$$
which we can rewrite as
$$
x^{q^2+1}(1+x^{\sqrt{2q}t})+x^{t}+1+x^{\sqrt{2q}t}+x^{(q-\sqrt{2q}+1)t}+x^t\sum_{j=0}^{\log \sqrt{q/2}-1} (x^{2^{j+1}(\sqrt{2q}-1)t}+x^{2^j(\sqrt{2q}-2)t})=0
$$
and dividing by $1+x^{\sqrt{2q}t}$ gives
$$
x^{q^2+1}+1+x^t \left(\frac{1+x^{(q-\sqrt{2q})t}}{1+x^{\sqrt{2q}t}}\right)+x^t\sum_{j=0}^{\log \sqrt{q/2}-1} x^{2^j(\sqrt{2q}-2)t}\left( \frac{x^{2^{j}(\sqrt{2q})t}+1}{1+x^{\sqrt{2q}t}}\right)=0. \qedhere
$$
\end{proof}

\begin{example} We describe some examples obtained from GAP \cite{GAP4}. For $q=8$ we have 
\[
Q_1(x)=x^{65}+x^{39}+x^{13}+1,
\]
and for $q=32$ we have 
\[
Q_1(x)=x^{1025}+x^{861}+x^{697}+x^{533}+x^{369}+x^{287}+x^{41}+1.
\]
\end{example}








{\bf Acknowledgments}
	
\medskip
We are indebted to Tao Feng for  bringing the work of Glauberman to our attention. Kanat Abdukhalikov was supported by UAEU grant G00004614. Duy Ho was supported by   the Tromsø Research Foundation (project “Pure Mathematics in
Norway”), and  UiT Aurora project MASCOT.
Simeon Ball and Tabriz Popatia were supported by the Spanish Ministry of Science, Innovation and Universities grants PID2020-113082GB-I00 and PID2023-147202NB-I00.

\end{document}